\documentclass[a4paper,twopage,reqno,12pt]{amsart}
\usepackage[top=30mm,right=30mm,bottom=30mm,left=30mm]{geometry}
\usepackage{graphicx}
\usepackage{amsmath}
\usepackage{amssymb}
\usepackage{amsfonts}
\usepackage{amsthm}
\usepackage{amstext}
\usepackage{amsbsy}
\usepackage{amsopn}
\usepackage{amscd}
\usepackage{enumerate}
\usepackage{multirow}
\usepackage{float}
\usepackage{color}
\usepackage{longtable}
\usepackage[colorlinks]{hyperref}
\usepackage[hyperpageref]{backref}
\usepackage{lscape}
\usepackage{longtable}
\newtheorem{theorem}{Theorem}[section]

\newtheorem{lemma}[theorem]{Lemma}
\newtheorem{proposition}[theorem]{Proposition}

\theoremstyle{definition}

\newtheorem{conjecture}{Conjecture}[section]
\theoremstyle{remark}
\newtheorem{remark}[theorem]{Remark}
\numberwithin{equation}{section}

\newcommand{\Irr}{\textsf{Irr}}
\newcommand{\Aut}{\textsf{Aut}}
\newcommand{\Out}{\mathsf{Out}}
\newcommand{\M}{\textsf{M}}

\newcommand{\ATLAS}{\textsf{ATLAS} }
\newcommand{\cd}{\mathsf{cd}}

\renewcommand{\leq}{\leqslant}

\newcommand{\Zbb}{\mathbb{Z}}

\providecommand{\mod}[1]{ \ (\mathrm{mod}\ #1)}

\newcommand{\midrule}{\hline}
\newcommand{\bottomrule}{\hline}

\begin{document}

\title[Groups with character degrees of almost simple sporadic groups]{ON GROUPS WITH THE SAME CHARACTER DEGREES AS ALMOST SIMPLE GROUPS WITH SOCLE SPORADIC SIMPLE GROUPS}

\author[S.H. Alavi]{Seyed Hassan Alavi}
 \address{
 S.H. Alavi, Department of Mathematics, Faculty of Science,
 Bu-Ali Sina University, Hamedan, Iran}
  \email{alavi.s.hassan@gmail.com (Gmail is preferred)}
 \email{alavi.s.hassan@basu.ac.ir}

\author[A. Daneshkhah]{Ashraf Daneshkhah$^{\ast}$}
 \thanks{$\ast$ Corresponding author: Ashraf Daneshkhah}
 \address{%
 A. Daneshkhah, Department of Mathematics,
 Faculty of Science,
 Bu-Ali Sina University, Hamedan, Iran}
 \email{daneshkhah.ashraf@gmail.com (Gmail is preferred)}
 \email{adanesh@basu.ac.ir}

\author[A. Jafari]{Ali Jafari}
 \address{%
 A. Jafari, Department of Mathematics,
 Faculty of Science,
 Bu-Ali Sina University, Hamedan, Iran}
 \email{a.jaefary@gmail.com}



\begin{abstract}
   Let $G$ be a finite group and $\cd(G)$ denote the set of complex irreducible character degrees of $G$. In this paper, we prove that if $G$ is a finite group and $H$ is an almost simple group whose socle is a sporadic simple group $H_{0}$ such that $\cd(G) =\cd(H)$, then $G'\cong H_{0}$ and there exists an abelian subgroup $A$ of $G$ such that $G/A$ is isomorphic to $H$. In view of Huppert's conjecture (2000), we also provide some examples to show that $G$ is not necessarily a direct product of $A$ and $H$, and hence we cannot extend this conjecture to almost simple groups.
\end{abstract}

\date{\today}
\maketitle

\section{Introduction}\label{sec:Intro}

Let $G$ be a finite group, and let $\Irr(G)$ be the set of complex irreducible character degrees of $G$. Denote the set of character degrees of $G$ by $\cd(G)=\{\chi(1)\vert \chi \in \Irr(G)\}$, and when the context allows us the set of irreducible character degrees will be referred to as the set of character degrees. There is growing interest in the information regarding the structure of $G$ which can be determined from the character degree set of $G$. It is well-known that the character degree set of $G$ can not use to completely determine the structure of $G$. For example, the non-isomorphic groups $D_8$ and $Q_8$ not only have the same set of character degrees, but also share the same character table.

The character degree set cannot be used to distinguish between solvable and nilpotent groups. For example,
if $G$ is either $Q_{8}$ or $S_{3}$, then $\cd(G) = \{1, 2\}$. However, in the late 1990s, Huppert \cite{Hupp-I}
posed a conjecture which, if true, would sharpen the connection between the character degree set of a non-abelian simple group and the structure of the group.

\begin{conjecture}[Huppert] Let $G$ be a finite group, and let $H$ be a finite non-abelian simple group such that the sets of character degrees of $G$ and $H$ are the same. Then $G \cong H \times A$, where $A$ is an abelian group.
\end{conjecture}

The conjecture asserts that the non-abelian simple groups
are essentially characterized by the set of their character degrees.
In addition to verifying this conjecture for many of the simple groups
of Lie type, it is also verified for all sporadic simple
groups \cite{ADTW-Fi23,ADTW-2013,HT}. Note that this conjecture does not extend to solvable groups, for example, $Q_{8}$ and $D_{8}$. We moreover cannot extend Huppert's conjecture to \emph{almost simple groups}. In fact, there are four groups $G$ of order $240$ whose character degrees are the same as $\Aut(A_{5})=S_{5}$. These groups are $SL_{2}(5)\ . \ \Zbb_{2}$ (non split), $SL_{2}(5):\Zbb_{2}$ (split), $A_{5}:\Zbb_{4}$ (split) and $S_{5}\times \Zbb_{2}$. If we further assume that $G'=A_{5}$, we still have two possibilities for $G$, namely, $A_{5}:\Zbb_{4}$ and $S_{5}\times \Zbb_{2}$. Indeed, the groups $A_{5}:\Zbb_{2^{n}}$, for $n\geq 1$, have the same character degree set as $S_{5}$. Although it is unfortunate to establish Huppert's conjecture for almost simple groups, we can prove the following result for finite groups whose character degrees are the same as those of almost simple groups with socle sporadic simple groups:

\begin{theorem}\label{thm:main}
Let $G$ be a finite group, and let $H$ be an almost simple group
whose socle $H_{0}$ is one of the sporadic simple groups.
If $\cd(G)=\cd(H)$, then $G'\cong H_{0}$ and $G/Z(G)$ is isomorphic to $H$.
\end{theorem}

In order to prove Theorem~\ref{thm:main}, we establish
the following steps introduced in \cite{Hupp-I}.
Let $H$ be an almost simple group with socle $H_{0}$, and
let $G$ be a group with the same character degrees as $H$. Then we show that
\begin{enumerate}[(1)]
  \item If $G'/M$ is a chief factor of $G$, then $G'/M$ is isomorphic to $H_{0}$;
  \item If $\theta \in \Irr(M)$ with $\theta(1)$, then $I_{G'}(\theta)=G'$ and so $M=M'$;
  \item $M=1$ and $G'\cong H_{0}$;
  \item $G/Z(G)$ is isomorphic to $H$.
\end{enumerate}

In Propositions~\ref{prop:spor-1-2}-\ref{prop:spor-5}, we will verify Steps 1-4, and the proof of Theorem~\ref{thm:main} follows immediately from these statements.

\begin{remark}\label{rem:main}
Recall that Theorem~\ref{thm:main} for the case where $H=H_{0}$
is a sporadic simple groups has already been settled,
see \cite{ADTW-Fi23,ADTW-2013,Hupp-II-VIII,HT}. Moreover,
if $H$ is the automorphism group of one of the Mathieu groups,
then Theorem \ref{thm:main} is also proved by the authors \cite{ADJ-2015}.
Therefore, we only need to focus on remaining cases where
$H=\Aut(H_{0})$ with $H_{0}$ one of
$J_{2}$, $HS$, $J_{3}$, $McL$, $He$, $Suz$, $O'N$, $Fi_{22}$,
$HN$ and $Fi_{24}'$.
\end{remark}

\section{Preliminaries}\label{sec:prem}

In this section, we present some useful results to prove Theorem~\ref{thm:main}. We first establish some definitions and notation.

Throughout this paper all groups are finite.
A  group $H$ is said to be an almost simple group
with socle $H_{0}$ if $H_{0}\leq H\leq \Aut(H_{0})$,
where $H_{0}$ is a non-abelian simple group.
If $N\unlhd G$ and $\theta\in \Irr(N)$, then the inertia group $I_G(\theta)$ of
$\theta$ in $G$ is defined by $I_G(\theta)=\{g\in G\ |\ \theta^g=\theta\}$. If the character $\chi=\sum_{i=1}^k e_i\chi_i$, where each $\chi_i$ is an irreducible character of $G$ and $e_i$ is a nonnegative integer, then those $\chi_i$ with $e_i>0$ are called the \emph{irreducible constituents} of $\chi$. The set of all irreducible constituents of $\theta^G$ is denoted by $\textrm{\Irr}(G|\theta)$. All further notation and definitions are standard and could be found in \cite{HuppBook,Isaacs-book}. For computation parts, we use \textsf{GAP}~\cite{GAP4}.

\begin{lemma}[{~\cite[Theorems 19.5 and 21.3]{HuppBook}}]\label{lem:gal}
Suppose $N\unlhd G$ and $\chi\in {\rm{\Irr}}(G)$.
\begin{enumerate}[(a)]
  \item If $\chi_N=\theta_1+\theta_2+\cdots+\theta_k$ with $\theta_i\in {\rm{\Irr}}(N)$, then $k$ divides $|G/N|$. In particular, if $\chi(1)$ is prime to $|G/N|$, then $\chi_N\in {\rm{\Irr}}(N)$.
  \item (Gallagher's Theorem) If $\chi_N\in {\rm{\Irr}}(N)$, then $\chi\psi\in {\rm{\Irr}}(G)$ for all $\psi\in {\rm{\Irr}}(G/N)$.
\end{enumerate}
\end{lemma}

\begin{lemma}[{~\cite[Theorems 19.6 and 21.2]{HuppBook}}]\label{lem:clif}
Suppose $N\unlhd G$ and $\theta\in {\rm{\Irr}}(N)$. Let $I=I_G(\theta)$.
\begin{enumerate}[(a)]
  \item  If $\theta^I=\sum_{i=1}^k\phi_i$ with $\phi_i\in {\rm{\Irr}}(I)$, then $\phi_i^G\in {\rm{\Irr}}(G)$. In particular, $\phi_i(1)|G:I|\in \cd(G)$.
  \item If $\theta$ extends to $\psi\in {\rm{\Irr}}(I)$, then $(\psi\tau )^G\in {\rm{\Irr}}(G)$ for all $\tau\in {\rm{\Irr}}(I/N)$. In particular, $\theta(1)\tau(1)|G:I|\in {\rm{\cd}}(G)$.
  \item If $\rho \in {\rm{\Irr}}(I)$ such that $\rho_N=e\theta$, then $\rho=\theta_0\tau_0$, where $\theta_0$ is a character of an irreducible projective representation of $I$ of degree $\theta(1)$ and $\tau_0$ is a character of an irreducible projective representation of $I/N$ of degree $e$.
\end{enumerate}
\end{lemma}

A character $\chi\in \textrm{\Irr}(G)$ is said to be \emph{isolated} in $G$ if $\chi(1)$ is divisible by no proper nontrivial character degree of $G$ and no proper multiple of $\chi(1)$ is a character degree of $G$. In this situation, we also say that $\chi(1)$ is an \emph{isolated degree} of $G$. We define a \emph{proper power} degree of $G$ to be a character degree of $G$ of the form $f^{a}$ for integers $f$ with $a>1$. 

\begin{lemma}[{~\cite[Lemma~3]{HT}\label{lem:factsolv}}] Let $G/N$ be a solvable factor group of
$G$ minimal with respect to being non-abelian. Then two cases can occur.
\begin{enumerate}[(a)]
  \item $G/N$ is an $r$-group for some prime $r$. In this case, $G$ has a proper prime power degree.
  \item $G/N$ is a Frobenius group with an elementary abelian Frobenius kernel $F/N$. Then $f:=|G:F|\in {\rm{\cd}}(G)$ and $|F/N|=r^a$ for some prime $r$ and $a$ is the smallest integer such that $r^a\equiv 1 \mod{f}$. 
  \begin{enumerate}[(1)]
     \item If $\chi\in {\rm{\Irr}}(G)$ such that no proper multiple of $\chi(1)$ is in ${\rm{\cd}}(G)$, then either $f$ divides $\chi(1)$, or $r^a$ divides $\chi(1)^2$.
     \item If $\chi\in {\rm{\Irr}}(G)$ is isolated, then $f=\chi(1)$ or $r^a\mid \chi(1)^2$.
  \end{enumerate}
\end{enumerate}
\end{lemma}

\begin{lemma}[{~\cite[Theorems~2-4]{Bia}}]\label{lem:exten} If $S$ is a non-abelian simple group, then there exists a nontrivial irreducible character $\theta$ of $S$
that extends to $\Aut(S)$. Moreover, the following holds:
\begin{enumerate}[(a)]
  \item if $S$ is an alternating group of degree at least $7$, then $S$ has two characters of consecutive degrees $n(n-3)/2$ and $(n-1)(n-2)/2$ that both extend to $\Aut(S)$.
  \item if $S$ is a simple group of Lie type, then the Steinberg character of $S$ of degree $|S|_p$ extends to $\Aut(S)$.
  \item if $S$ is a sporadic simple group or the Tits group, then $S$ has two nontrivial irreducible characters of coprime degrees which both extend to $\Aut(S)$.
\end{enumerate}
\end{lemma}

\begin{lemma}[{~\cite[Lemma 5]{Bia}}] \label{lem:exten-2}
Let $N$ be a minimal normal subgroup of $G$ so that $N\cong S^k$, where $S$ is a non-abelian simple group. If
$\theta\in \Irr(S)$ extends to $\Aut(S)$, then $\theta^k\in \Irr(N)$ extends to $G$.
\end{lemma}

\begin{lemma}[{~\cite[Lemma 6]{Hupp-I}}]\label{lem:schur}
Suppose that $M\unlhd G'=G''$ and for every $\lambda\in {\rm{\Irr}}(M)$ with $\lambda(1)=1$, $\lambda^g=\lambda$ for all $g\in G'$. Then $M'=[M,G']$ and $|M/M'|$ divides the order of the Schur multiplier of $G'/M$.
\end{lemma}

\begin{lemma}[{~\cite[Theorem D]{Moreto}}]\label{lem:sol}
Let $N$ be a normal subgroup of a finite group $G$ and let $\varphi \in \Irr(N)$ be
$G$-invariant. Assume that $\chi(1)/\varphi(1)$ is odd, for all $\chi(1)\in \Irr(G|\varphi)$. Then $G/N$ is solvable.
\end{lemma}

\section{Proof of the main result}\label{sec:proof}


In this section, we prove Theorem~\ref{thm:main} for almost simple group $H$ whose socle is a sporadic simple group $H_{0}$ as in Remark~\ref{rem:main}. For convenience, we first mention some properties of $H$ and $H_{0}$ which can be drawn from \textsf{ATLAS} \cite{Atlas}.

\begin{lemma}\label{lem:spor}
Suppose that $H_{0}$ is one of the sporadic simple groups as in the first column of {\rm Table~\ref{tbl:spor}}, and suppose that $H=\Aut(H_{0})$.
Then
\begin{enumerate}[(a)]
  \item the outer automorphism group $\Out(H_{0})$ of $H_{0}$ is isomorphic to $\Zbb_{2}$, and the Schur multiplier $\M(H_{0})$ of $H_{0}$ is listed in {\rm Table~\ref{tbl:spor}};
  \item $H$ has neither consecutive, nor proper power degrees;
  \item if $K$ is a maximal subgroup of $H_{0}$ whose index in $H_{0}$ divides some degrees $\chi(1)$ of $H$, then $K$ is given in {\rm Table~\ref{tbl:spor}}, and for each $K$, $\chi(1)/|H_{0}:K|$ divides $t(K)$ as in {\rm Table~\ref{tbl:spor}}.
\end{enumerate}
\end{lemma}
\begin{proof}
Parts (a) and (b) follows from \ATLAS \cite{Atlas}, and part (c) is a straightforward calculation.
\end{proof}
\begin{table}[t]
  \caption{Some properties of some sporadic simple groups $H_{0}$ and their automorphism groups.\label{tbl:spor}}
  \begin{tabular}{llclp{6.1cm}}
  \hline
  \multicolumn{1}{c}{$H_{0}$} &
  \multicolumn{1}{c}{$\Aut(H_{0})$} &
  \multicolumn{1}{c}{$M(H_{0})$} &
  \multicolumn{1}{c}{$K$} &
  \multicolumn{1}{c}{$t(K)$} \\
  \midrule
    $J_2$ & $J_2:2$ & $\Zbb_{2}$ & $U_{3}(3)$ & 3 \\
    %
    $HS$ & $HS:2$ & $\Zbb_{2}$ &  $M_{22}$ & $2^{5}$ \\
         & & &  $U_{3}(5):2$ & $6$ or $8$ \\
    $J_3$ & $J_3:2$ & $\Zbb_{3}$ & - & - \\
    %
    $McL$ & $McL:2$& $\Zbb_{3}$ & $U_{4}(3)$ & $5\cdot 7$ or $2^{2}\cdot 3\cdot 5$ \\
    %
    $He$ & $He:2$ & $1$ &  $S_{4}(4):2$ & 1 \\
    %
    $Suz$ & $Suz:2$ & $\Zbb_{6}$ & $G_{2}(4)$ & $3^{2}\cdot 13$ or $3\cdot 5\cdot 7$ \\
    %
     & & &  $U_{5}(2)$ & $5$\\
    $O'N$ & $O'N:2$ & $\Zbb_{3}$ &  - & - \\
    %
    $Fi_{22}$ & $Fi_{22}:2$ & $\Zbb_{6}$ & $2\cdot U_6(2)$ & $2^2\cdot3\cdot5\cdot11$, $2^4\cdot5\cdot7$, $3\cdot5\cdot11$ or $3^5$ \\
    %
     & &  &   $O^+_8(2):S_3$ & $6$ \\
     & &  &   $2^{10}:M_{22}$ & $6$ \\
    $HN$ &  $HN:2$  & $1$ & - & -\\ 
    $Fi'_{24}$ & $Fi'_{24}:2$  & $\Zbb_{3}$ & $2.Fi_{23}$ & $2^4\cdot5^2\cdot7\cdot17\cdot23$, $2\cdot3^3\cdot7\cdot11\cdot13\cdot17$, $2^2\cdot3\cdot11\cdot13\cdot17\cdot23$, $2^3\cdot3\cdot7\cdot11\cdot13\cdot23$, $2^4\cdot3\cdot13\cdot17\cdot23$, $2^2\cdot7\cdot11\cdot17\cdot23$, $11\cdot13\cdot17\cdot23$ \\
    %
  %
  \bottomrule
  \multicolumn{5}{p{12cm}}{The symbol `-' means that there is no subgroup $K$ satisfying the conditions in Lemma~\ref{lem:spor}(c).}\\
  \end{tabular}
\end{table}

\begin{proposition}\label{prop:spor-list}
Let $S$ be a sporadic simple group or the Tits group $^{2}F_{4}(2)'$ whose character degrees divide some degrees of an almost simple group with socle a sporadic simple group $H_{0}$. Then either $S$ is isomorphic to $H_{0}$, or $(H,S)$ is as in {\rm Table~\ref{tbl:simple}}.
\end{proposition}
\begin{proof}
The proof follows from \cite{Atlas}, see also  \cite[Proposition 3.1]{ADJ-2015}.
\end{proof}
\begin{table}[t]
  \caption{Sporadic simple groups $S$ and the Tits group whose irreducible character degrees divide some character degrees of almost simple groups $H$ with socle sporadic simple groups.\label{tbl:simple}}
  \begin{tabular}{lp{9.8cm}}
  \hline
  \multicolumn{1}{c}{$H$} & \multicolumn{1}{c}{$S$} \\
  \midrule
  %
    %
    $M_{12}$, $M_{12}:2$ &
    $M_{11}$, $M_{12}$ \\
    %
    %
    $M_{23}$ &
    $M_{11}$, $M_{23}$ \\
    $M_{24}$ &
    $M_{11}$, $M_{24}$ \\
    %
    %
    %
    %
    $J_4$ &
    $M_{11}$, $M_{12}$, $M_{22}$, $J_4$ \\
    $HS$, $HS:2$&
    $M_{11}$, $M_{22}$, $HS$ \\
    $McL$, $McL:2$&
    $M_{11}$, $McL$ \\
    $Suz$, $Suz:2$ &
    $M_{11}$, $M_{12}$, $M_{22}$, $J_2$, $Suz$, $^{2}F_{4}(2)'$ \\
    $Co_3$ &
    $M_{11}$, $M_{12}$, $M_{22}$, $M_{23}$, $M_{24}$, $Co_3$ \\
    $Co_2$ &
    $M_{11}$, $M_{12}$, $M_{22}$, $M_{23}$, $M_{24}$ $J_2$, $Co_2$\\
    $Co_1$ &
    $M_{11}$, $M_{12}$, $M_{22}$, $M_{23}$, $M_{24}$, $McL$,
    $J_2$, $HS$, $Co_1$, $Co_{3}$, $^{2}F_{4}(2)'$ \\
    %
    %
    $Fi_{22}$,$Fi_{22}:2$ &
    $M_{11}$, $M_{12}$, $M_{22}$, $J_2$,  $Fi_{22}$\\
    $Fi_{23}$ &
    $M_{11}$, $M_{12}$, $M_{22}$, $M_{23}$, $M_{24}$, $HS$, $J_2$, $Fi_{23}$, $^{2}F_{4}(2)'$\\
    $Fi'_{24}$, $Fi'_{24}:2$  &
    $M_{11}$, $M_{12}$, $M_{22}$, $M_{23}$, $M_{24}$, $He$, $J_2$, $Fi'_{24}$, $^{2}F_{4}(2)'$\\
    $Th$ &
    $J_2$, $Th$, $^{2}F_{4}(2)'$ \\
    $Ru$ &
    $J_2$, $Ru$, $^{2}F_{4}(2)'$ \\
    $Ly$  &
    $M_{11}$, $M_{12}$, $J_2$, $Ly$\\
    $HN$,  $HN:2$  &
    $M_{11}$, $M_{12}$, $M_{22}$, $J_1$, $J_{2}$, $HS$, $HN$ \\
    $O'N$, $O'N:2$ &
    $M_{11}$, $M_{12}$, $M_{22}$, $O'N$ \\
    $B$ &
    $M_{11}$, $M_{12}$, $M_{22}$, $M_{23}$, $M_{24}$, $J_1$, $J_2$ , $J_3$, $HS$, $McL$, $Suz$, $Fi_{22}$, $Co_3$, $Co_2$, $Th$, $B$, $^{2}F_{4}(2)'$ \\
    $M$ &
    $M_{11}$, $M_{12}$, $M_{22}$, $M_{23}$, $M_{24}$, $J_1$, $J_2$, $J_3$, $HS$, $McL$, $Suz$, $Fi_{22}$, $Co_3$, $Co_2$, $He$, $O'N$, $Ru$, $M$, $^{2}F_{4}(2)'$ \\
  \bottomrule
  \end{tabular}
\end{table}
\begin{table}[t]
  \caption{Some isolated degrees of some automorphism groups of sporadic simple groups.\label{tbl:spor-iso}}
  \begin{tabular}{lp{3.5cm}p{3.5cm}p{3.5cm}}
  \hline
  \multicolumn{1}{c}{$H$} & \multicolumn{1}{c}{$\chi_{1}(1)$} &\multicolumn{1}{c}{$\chi_{2}(1)$} &\multicolumn{1}{c}{$\chi_{3}(1)$}\\
  \midrule
  $HS:2$ & $825=3\cdot 5^2\cdot 11$ & $1792=2^8\cdot 7$ & $2520=2^3\cdot 3^2\cdot 5\cdot 7$ \\
  $J_{3}:2$ & $170=2\cdot 5\cdot 17$ & $324=2^2\cdot 3^4$ & $1215=3^5\cdot 5$ \\
  $McL:2$ & $1750=2\cdot 5^3\cdot 7$ & $4500=2^2\cdot 3^2\cdot 5^3$ & $5103=3^6\cdot 7$ \\
    $He:2$ & $1920=2^7\cdot 3\cdot 5$ & $2058=2\cdot 3\cdot 7^3$ & $20825=5^2\cdot 7^2\cdot 17$ \\
    $O'N:2$ & $10944=2^6\cdot 3^2\cdot 19$ & $26752=2^7\cdot 11\cdot 19$ & $116963=7^3\cdot 11\cdot 31$ \\
    $Fi_{22}:2$ & $360855=3^8\cdot 5\cdot 11$ & $577368=2^3\cdot 3^8\cdot 11$ & $1164800=2^9\cdot 5^2\cdot 7\cdot 13$ \\
    $HN:2$ & $1575936=2^{10}\cdot 3^4\cdot 19$ & $2784375=3^4\cdot 5^5\cdot 11$ & $3200000=2^{10}\cdot 5^5$ \\
    $Fi_{24}':2$ & $159402880=2^7\cdot 5\cdot 7^2\cdot 13\cdot 17\cdot 23$ & $5775278080=2^{14}\cdot 5\cdot 11\cdot 13\cdot 17\cdot 29$ & $156321775827=3^{14}\cdot 7^2\cdot 23\cdot 29$ \\
  \bottomrule
  \end{tabular}
\end{table}

\begin{proposition}\label{prop:spor-1-2}
Let $G$ be a finite group, and let $H$ be an almost simple group whose socle is a sporadic simple group $H_{0}$. If $\cd(G)=\cd(H)$, then the chief factor $G'/M$ of $G$ is isomorphic to $H_{0}$.
\end{proposition}
\begin{proof}
We first apply Remark~\ref{rem:main}, and so we may assume that  $H=\Aut(H_{0})$, where $H_{0}$ is one of the sporadic groups $J_{2}$, $J_{3}$, $McL$, $HS$, $He$, $HN$, $Fi_{22}$, $Fi_{24}'$, $O'N$ and $Suz$.

We now prove that $G'=G''$. Assume the contrary. Then there is a normal subgroup $N$ of $G$, where $N$ is a maximal such that $G/N$ is a non-abelian solvable group. Now we apply Lemma~\ref{lem:factsolv}, and since $G$ has no prime power degree, $G/N$ is a Frobenius group with kernel $F/N$ of order $r^a$. In this case, $1<f=|G : F| \in \cd(G)$.

Suppose that $H_{0}$ is not $J_{2}$ and $Suz$. Then $G$ has three isolated coprime degrees as in Table~\ref{tbl:spor-iso}. Now Lemma~\ref{lem:factsolv}(b.2) implies that $f$ must be equal to these degrees, which is impossible.

Suppose $H_{0}=Suz$.  Let $r=2$. Note that $r^{a}-1=2^{a}-1$ is less than the smallest nontrivial degree $143$ of $G$,   for $1\leq a\leq 4$. Then by Lemma~\ref{lem:factsolv}(b.2), $f$ must divide both degrees $75075=3\cdot 5^2\cdot 7\cdot 11\cdot 13$ and $5940=2^2\cdot 3^3\cdot 5\cdot 11$, and so $f$ divides $3\cdot 5\cdot 11$, but none of divisors of $3\cdot 5\cdot 11$ is a degree of $G$. Therefore, $r\neq 2$. Now we apply Lemma~\ref{lem:factsolv}(b.2) to isolated degrees  $66560=2^{10}\cdot 5\cdot 13$ and $133056=2^6\cdot 3^3\cdot 7\cdot 11$,  and since $r\neq 2$, it follows that $f$ must be equal to both of these degrees, which is impossible.

Suppose finally $H_{0}=J_{2}$. Here we make the same argument as in the case of $Suz$. If $r=5$, then by Lemma~\ref{lem:factsolv}(b.2), $f$ must divide both $2^{5}\cdot 3^{2}$ and $2\cdot 3^{3}\cdot 7$, and so $f$ divides $2\cdot 3^{2}$, but $G$ has no degree as a divisor of $2\cdot 3^{2}$. Thus $r\neq 5$. Now we apply Lemma~\ref{lem:factsolv}(b.2) to isolated degrees $2^{5}\cdot 5$ and $5^{2}\cdot 7$, and so $f$ must be equal to both of these degrees, which is impossible. \smallskip

In conclusion, $G'=G''$. Let now  $G'/M$ be a chief factor of $G$. As $G'$ is perfect, $G'/M$ is non-abelian, and so  $G'/M$ is isomorphic to $S^k$ for some non-abelian simple group $S$ and some integer $k\geq 1$.

We first show that $k=1$. Assume the contrary. Then by Lemma~\ref{lem:exten}, $S$ possesses a nontrivial irreducible character $\theta$ extendible to $\Aut(S)$, and so
Lemma \ref{lem:exten-2} implies that $\theta^k\in \Irr(G'/M)$ extends to $G/M$, that is to say, $G$ has a proper power degree contradicting Lemma~\ref{lem:spor}(b). Therefore, $k=1$, and hence $G'/M\cong S$.

If $S$ is an alternating group of degree $n\geq 7$. By Lemma~\ref{lem:exten}(a), $S$ has nontrivial irreducible characters $\theta_1$ and $\theta_2$ with $\theta_1(1)=n(n-3)/2$ and $\theta_2(1)=\theta_1(1)+1=(n-1)(n-2)/2$, respectively, and both $\theta_i$ extend to $\Aut(S)$. Thus $G$ possesses two consecutive nontrivial character degrees, contradicting Lemma~\ref{lem:spor}(b).

If $S\neq {}^2F_4(2)'$ is a simple group of Lie type in
characteristic $p$, then the Steinberg character of $S$ of degree $|S|_p$ extends to $\Aut(S)$ so that $G$ possesses a nontrivial prime power degree contradicting Lemma~\ref{lem:spor}(b).

If $S$ is a sporadic simple group or the Tits group ${}^2F_4(2)'$, then irreducible character degrees of $S$ divide some degrees of $H$, and so by Proposition~\ref{prop:spor-list}, $S\cong H_{0}$ or $(H,S)$ is as in Table~\ref{tbl:simple}. In the later case, for a given $H$ as in the first row of Table~\ref{tbl:simple}, assume that $S$ is not isomorphic to $H_{0}$. Then we apply Lemma~\ref{lem:exten}(c), and so, for each $S$ as in the first row of Table~\ref{tbl:spor-exten}, $G$ possesses an irreducible character of degree listed in the second row of Table~\ref{tbl:spor-exten}. This leads us to a contradiction. Therefore, $S\cong H_{0}$, and hence $G'/M$ is isomorphic to $H_{0}$.
\end{proof}

\begin{table}
  \caption{Some degrees of some sporadic simple groups $S$ and the Tits group.}\label{tbl:spor-exten}
  \centering
  \begin{tabular}{lcccccccccc}
    \hline
    \multicolumn{1}{l}{$S$} & $M_{11}$ & $M_{12}$ & $M_{22}$ & $M_{23}$ & $M_{24}$ & $J_{1}$ & $J_{2}$ & $HS$ & $He$ & $^{2}F_{4}(2)'$\\
    \midrule
    \multicolumn{1}{l}{Degree} & $10$ & $54$ & $21$ & $22$ & $23$ & $76$ & $36$ & $22$ & $1275$ & $27$\\
    \bottomrule
    \end{tabular}
\end{table}

\begin{proposition}\label{prop:spor-3}
Let $G$ be a finite group with $\cd(G)=\cd(H)$ where $H$ is an almost simple group whose socle is a sporadic simple group $H_{0}$. Let also the chief factor $G'/M$ be isomorphic to $H_{0}$. If $\theta \in \Irr(M)$ with $\theta(1)=1$, then  $I_{G'}(\theta)=G'$.
\end{proposition}
\begin{proof}
By Remark~\ref{rem:main}, we may assume that  $H=\Aut(H_{0})$, where $H_{0}$ is one of the sporadic groups $J_{2}$, $J_{3}$, $McL$, $HS$, $He$, $HN$, $Fi_{22}$, $Fi_{24}'$, $O'N$ and $Suz$.

Suppose $I=I_{G'}(\theta)<G'$. Let $\theta^I= \sum_{i=1}^{k} \phi_i$, where $\phi_i \in \Irr(I)$ for $ i=1,2,...,k$. Assume that $U/M$ is a maximal subgroup of $G'/M\cong H_{0}$ containing $I/M$ and set $t:= |U:I|$. It follows from Lemma \ref{lem:clif}(a) that $\phi_i(1)|G':I| \in \cd(G')$, and so $t\phi_i(1)|G':U|$ divides some degrees of $G$. Then $|G':U|$ must divide some character degrees of $G$, and hence for each $H_{0}$ as in the first column of Table~\ref{tbl:spor},  by Lemma~\ref{lem:spor}(c), $U/M$ can be the subgroup $K$ listed in the fifth column of Table~\ref{tbl:spor} and $t\phi_i(1)|G':U|$ must divide the positive integers $t(K)$ mentioned in the sixth column of Table~\ref{tbl:spor}.

If $H_{0}$ is $J_{3}$, $O'N$ or $HN$, then by Lemma~\ref{lem:spor}(b), there is no such subgroup $U/M$, and so $I_{G'}(\theta)=G'$ in these cases. We now discuss each remaining case separately. \smallskip

\noindent \textbf{(1)} $H_{0}=J_{2}$. Then by Lemma~\ref{lem:spor}(b), $U/M \cong U_3(3)$ and $t\phi_i(1)$ divides $3$, for all $i$. Since $U_{3}(3)$ has no any subgroup of index $3$ \cite[p. 14]{Atlas}, it follows that $t=1$, that is to say, $I/M= U/M\cong U_3(3)$. Since also $U_3(3)$ has trivial Schur multiplier, it follows from \cite[Theorem 11.7]{Isaacs-book} that $\theta$ extends to $\theta _0\in \Irr(I)$, and so by Lemma \ref{lem:clif}(b), $(\theta_0\tau)^{G'}\in \Irr(G')$, for all $\tau \in \Irr(I/M)$. For $\tau(1)=27\in \cd(U_3(3))$, it turns out that $3\cdot27\cdot\theta_0(1)$ divide some degrees of $G$, which is a contradiction. Therefore, $\theta$ is $G'$-invariant.\smallskip

\noindent \textbf{(2)} $H_{0}=HS$. Then by Lemma~\ref{lem:spor}(b), one of the following holds:
\begin{enumerate}
  \item[(i)] $U/M\cong M_{22}$ and $t\phi_i(1)$ divides $2^5$, for all $i$.
\end{enumerate}
As $U/M \cong M_{22}$ does not have any subgroup of index $2^m$ for $m=1,...,5$, by  ~\cite[pp. 80-81]{Atlas}, $t=1$ and $I/M= U/M\cong M_{22}$ and $\phi_i(1)$ divides $2^5$. Assume first that $e_j=1$ for some $j$. Then $\theta$ extends to $\varphi_{j}\in \Irr(I)$. By Lemma~\ref{lem:gal}(b), $\tau \varphi_{j}$ is an irreducible constituent of $\theta^I$ for every $\tau\in \Irr(I/M)$, and so $\tau(1) \varphi_{j}(1)=\tau(1)$ divides $2^5$. Now we choose $\tau\in \Irr(I/M)=\Irr(M_{22})$ with $\tau(1)=21$ and this degree does not divide $2^5$, which is a contradiction. Therefore $e_i > 1$ for all $i$. We deduce that, for each $i$, $e_i$ is the degree of a nontrivial proper irreducible projective representation of $M_{22}$. As $\phi_i(1)=e_i\theta(1)=e_i$, each $e_i$ divides $2^5$. It follows that $e_i\leq 2^5$ for each $i$ and $e_i$ is the degree of a nontrivial proper irreducible projective representation of $M_{22}$, but according to~\cite[pp. 39-41]{Atlas}, there is no such a projective degree.
\begin{enumerate}
  \item[(ii)] $U/M\cong U_3(5):2$ and $t\phi_i(1)$ divides $6$ or $8$, for all $i$.
\end{enumerate}
Let $M \leqslant W \leqslant U$ such that $W/M\cong U_3(5)$. Then $W\unlhd U$. Assume that $W \nleqslant I$. Since $t=|U:I|=|U:WI|\cdot|WI:I|$ and $|WI:I|=|W :W∩I|$, the index of some maximal subgroup of $W/M\cong U_3(5)$ divides $t$ and so divides $6$ or $8$, which is a contradiction by ~\cite[pp. 34-35]{Atlas}. Thus $W \leqslant I \leqslant U$. Let $M \leqslant V \leqslant W$ such that $V/M \cong M_{10}$. We have that $\theta$ is $V$-invariant and, since the Schur multiplier of $V/M$ is trivial, $\theta$ extends to $\theta_0 \in \Irr(V)$. By Lemma~\ref{lem:gal}(b), $\tau \theta_0$ is an irreducible
constituent of $\theta^V$ for every $\tau\in \Irr(V/M)$. Choose $\tau\in \Irr(V/M)$ with $\tau(1)=16$ and let $\gamma = \tau \theta_0\in Irr(V|\theta)$. If $\chi \in \Irr(I)$ is an irreducible constituent of $\gamma^I$, then $\chi(1) \geqslant \gamma(1)$ by Frobenius reciprocity ~\cite[Lemma 5.2]{Isaacs-book} and also $\chi(1)$ divides $6$ or $8$,  which implies that $16=\gamma(1)\leqslant \chi(1) \leqslant 8$, which is contradiction. \smallskip

\noindent \textbf{(3)} $H_{0}=McL$. Then by Lemma~\ref{lem:spor}(b), $U/M\cong U_4(3)$ and $t\phi_i(1)$ divides $35$ or $60$, for all $i$. By inspecting the list of maximal subgroups of $U/M \cong U_4(3)$ in ~\cite[pp. 52-59]{Atlas}, no index of maximal subgroup of $U_4(3)$ divides $35$ or $60$, and so $t=1$. Thus $I/M= U/M\cong U_4(3)$ and $\phi_i(1)$ divides $35$ or $60$, for all $i$. Assume first $e_j=1$, for some $j$. Then $\theta$ extends to $\varphi_{j}\in \Irr(I)$. It follows from Lemma~\ref{lem:gal}(b) that  $\tau \varphi_{j}$ is an irreducible constituent of $\theta^I$ for every $\tau\in \Irr(I/M)$, and so $\tau(1) \varphi_{j}(1)=\tau(1)$ divides $35$ or $60$. Now let $\tau\in \Irr(I/M)=\Irr(U_{4}(3))$ with $\tau(1)=21$ and this degree does not divide neither $35$, nor $60$, which is a contradiction. Therefore $e_i > 1$, for all $i$. Therefore, for each $i$, $e_i$ is the degree of a nontrivial proper irreducible projective representation of $U_{4}(3)$. As $\phi_i(1)=e_i\theta(1)=e_i$, each $e_i$ divides $35$ or $60$, it follows from \cite[pp. 53-59]{Atlas} that $e_{i}\in\{6,15,20,35\}$. Let now $M \leqslant V \leqslant U$ such that $V/M \cong U_3(3)$. We have that $\theta$ is $V$-invariant, and since the Schur multiplier of $V/M$ is trivial, $\theta$ extends to $\theta_0 \in \Irr(V)$. It follows from  Lemma~\ref{lem:gal}(b) that $\tau \theta_0$ is an irreducible constituent of $\theta^V$ for every $\tau\in \Irr(V/M)$. Take $\tau\in \Irr(V/M)$ with $\tau(1)=32$, and let $\gamma = \tau \theta_0\in Irr(V|\theta)$. If $\chi \in \Irr(I)$ is an irreducible constituent of $\gamma^I$, then $\chi(1) \geq \gamma(1)=32$ by Frobenius reciprocity ~\cite[Lemma 5.2]{Isaacs-book}. This shows that $e_{i}=35$, for all $i$, that is to say, $\varphi_{i}(1)/\theta(1)$ divides $35$, for all $i$, and so Lemma~\ref{lem:sol} implies that $I/M\cong U_{4}(3)$ is solvable, which is a contradiction. \smallskip

\noindent \textbf{(4)} $H_{0}=He$. Then by Lemma~\ref{lem:spor}(b), $U/M\cong S_4(4):2$ and $t=1$, or equivalently, $I/M=U/M\cong S_4(4):2$. Moreover, $\phi_i(1)=1$, for all $i$. Then $\theta$ extends to $\phi_i \in \Irr(I)$, and so by Lemma~\ref{lem:clif}(b), $2058\tau (1)$ divides some degrees of $G$, for $\tau (1)=510$, which is a contradiction. Therefore $I_{G'}(\theta)=G'$.\smallskip

\noindent \textbf{(5)} $H_{0}=Suz$. Then by Lemma~\ref{lem:spor}(b), one of the following holds:
\begin{enumerate}
  \item[(i)] $U/M\cong G_2(4)$ and $t\phi_i(1)$ divides $3^2\cdot13$ or $3\cdot5\cdot7$, for all $i$.
\end{enumerate}
By inspecting the list of maximal subgroups of $G_2(4)$ in \cite[pp. 97-99]{Atlas}, no index of a maximal subgroup of $G_2(4)$ divides $3^2\cdot13$ or $3\cdot5\cdot7$,  then $t = 1$, and so $I/M= U/M\cong G_2(4)$. Note that $\phi_i(1)/\theta(1)$ divides $3^2\cdot13$ or $3\cdot5\cdot7$, for all $i$. If $\phi_i(1)/\theta(1)>1$, for all $i$, then we apply Lemma~\ref{lem:sol}, and so we conclude that $I/M$ is solvable, which is a contradiction. Therefore, $\varphi_{i}(1)=\theta(1)=1$ in which case $\theta$ extends to $\varphi_{i}$, for some $i$. It follows from  Lemma~\ref{lem:gal}(b) that $\tau \varphi_{i}$ is an irreducible
constituent of $\theta^I$ for every $\tau\in \Irr(I/M)$, and then $\tau(1) \varphi_{i}(1)=\tau(1)$ divides $3^2\cdot13$ or $3\cdot5\cdot7$. We can choose $\tau\in \Irr(I/M)=\Irr(G_2(4))$ with $\tau(1)=65$ and this degree does not divide $3^2\cdot13$ or $3\cdot5\cdot7$, which is a contradiction.
\begin{enumerate}
  \item[(ii)] $U/M\cong U_5(2)$  and $t\phi_i(1)$ divides $5$, for all $i$.
\end{enumerate}
As $U/M \cong U_5(2)$ does not have any subgroup of index $5$, by  ~\cite[pp. 72-73]{Atlas},$t=1$ and so $I/M= U/M\cong U_5(2)$. Thus $\phi_i(1)/\theta(1)$ divides $5$, for all $i$. Since $U_5(2)$ has trivial Schur multiplier, it follows that $\theta$ extends to $\theta _0\in \Irr(I)$, and so by Lemma \ref{lem:clif}(b) $(\theta_0\tau)^{G'}\in \Irr(G')$, for all $\tau \in \Irr(I/M)$. For $\tau(1)=300\in \cd(U_5(2))$, it turns out that $5\cdot300\cdot\theta_0(1)=2^2\cdot3\cdot5^3$ divides some degrees of $G$, which is a contradiction.\smallskip

\noindent \textbf{(6)} $H_{0}=Fi_{22}$. Then by Lemma~\ref{lem:spor}(b),  one of the following holds:
\begin{enumerate}
  \item[(i)] $U/M \cong2\cdot U_6(2)$ and $t\varphi_i(1)$ divides one of $3\cdot5\cdot11$, $2^2\cdot3\cdot5\cdot11$, $2^4\cdot5\cdot7$ or $3^5$, for all $i$.
\end{enumerate}
\noindent As $U/M$ is perfect, the center of $U/M$ lies in every maximal subgroup of $U/M$ and so the indices of maximal subgroups of $U/M$ and those of $U_6(2)$ are the same. By inspecting the list of maximal subgroups of $U_6(2)$ in ~\cite[pp. 115-121]{Atlas}, the index of a maximal subgroup of $U_6(2)$ no divides  $3\cdot5\cdot11$, $2^2\cdot3\cdot5\cdot11$, $2^4\cdot5\cdot7$ or $3^5$. Thus $t=1$ and hence $I=U$. Let $M \leq L \leq I$ such that $L/M$ is isomorphic to the center of $I/M$ and let $\lambda \in \Irr(L|\theta)$. As $L \unlhd I$, for any $\varphi\in \Irr(I|\lambda)$ we have that $\varphi(1)$ divides $3\cdot5\cdot11$, $2^2\cdot3\cdot5\cdot11$, $2^4\cdot5\cdot7$ or $3^5$. As above, we deduce that $\lambda$ is $I$-invariant.
Let $L\leq T \leq I$ such that $T/L \cong U_5(2)$. It follows that $\lambda$ is $T$-invariant and since the Schur multiplier of $T/L\cong U_5(2)$ is trivial, we have that $\lambda$ extends
to $\lambda_0 \in \Irr(T)$. By ~\ref{lem:gal}(b), $\tau\lambda_0$ is an irreducible constituent of $\lambda^T$ for every $\tau\in \Irr(T/L)$. Choose $\tau\in\Irr(T/L)$ with $\tau(1)=2^{10}$ and let $\gamma=\tau\lambda_0\in \Irr(T|\lambda)$. If $\chi\in\Irr(I)$ is any irreducible constituent of $\gamma^I$, then $\chi(1)\geq\gamma(1)$ by Frobenius reciprocity ~\cite[Lemma 5.2]{Isaacs-book} and $\chi(1)$ divides $3\cdot5\cdot11$, $2^2\cdot3\cdot5\cdot11$, $2^4\cdot5\cdot7$ or $3^5$, which implies that $\gamma(1)2^{10}\lambda(1)\leq \chi(1)\leq660$, which is impossible.
\begin{enumerate}
  \item[(ii)] $U/M \cong O^+_8(2):S_3$ and $t \varphi_i(1)$ divides $6$, for all $i$.
\end{enumerate}
Let $M \unlhd W \unlhd U$ such that $W/M \cong O^+_8(2)$. We have that $M \unlhd I∩W \unlhd I$ and $M \unlhd I∩W \leqslant W$. Assume $W \nless I$. Then $I \lneq WI \leqslant U$ and $t=|U:I|=|U : WI|\cdot|WI : I|$. Now $|WI:I|=|W:W ∩ I|>1$, and hence $t$ is divisible by $|W:W \cap I|$. As $W/M \cong O^+_8(2)$, $t$ is divisible by the index of some maximal subgroup of $O^+_8(2)$. Thus some index of a maximal subgroup of $O^+_8(2)$ divides $6$, which is impossible by ~\cite[pp. 85-88]{Atlas}. Thus $W \leq I \leq U$. Write $\theta ^W =\sum_{i=1}^{l} f_i \mu_i$ where $\mu_i \in \Irr(W|\theta)$ for $ i=1,2,...,l$. As $W \unlhd I$, $\mu_i(1)$ divides $6$ for every $i$. If $f_j=1$ for some $j$,
then $\theta$ extends to $\theta_0\in\Irr(W)$. By ~\ref{lem:gal}(b), $\tau \theta_0$ is an irreducible constituent of $\theta^W$ for every $\tau \in \Irr(W/M)$, and so $\tau(1)\theta_0(1)=\tau(1)$ divides
$6$. However we can choose $\tau \in \Irr(W/M)$ with $\tau(1)=28$ and this degree does not divide $6$. Therefore $f_i>1$, for all $i$. We deduce that, for each $i$, $f_i$ is the degree of a nontrivial proper irreducible projective representation of $O^+_8(2)$. As $\mu_i(1)=f_i \theta(1)=f_i$, each $f_i$ divides $6$. This is impossible as the
smallest nontrivial proper projective degree of $O^+_8(2)$ is $8$.

\begin{enumerate}
  \item[(iii)]$U/M \cong 2^{10}:M_{22}$ and $t\varphi_i(1)$ divides $6$, for all $i$.
\end{enumerate}
Let $M \unlhd L \unlhd U$ such that $L/M \cong 2^{10}$. We have that $L\unlhd U$ and $U/L \cong M_{22}$. The same argument as in part (ii) shows that $U=IL$ since the minimal index of a maximal subgroup of $M_{22}$ is $22$ by ~\cite[pp. 39-41]{Atlas}. Hence $U/L \cong I/L_1\cong M_{22}$, where $L_1=L \cap I \unlhd I$. Let $\lambda \in \Irr(L_1|\theta)$. Then for any $\varphi \in \Irr(I|\lambda)$, we have that $\varphi(1)$ divides $6$. We conclude that $\lambda$ is $I$-invariant as the index of a maximal subgroup of $I/L_1\cong M_{22}$ is at least $22$. Write $\lambda^I=\sum_{i=1}^{l} f_i\mu_i$, where $\mu_i \in \Irr(I|\lambda)$ for $ i=1,2,...,l$. Then $\mu_i(1)$ divides $6$, for each $i$. If $f_j=1$ for some $j$, then $\lambda$ extends to $\lambda_0 \in \Irr(I)$. By ~\ref{lem:gal}(b), $\tau \lambda_0$ is an irreducible constituent of $\lambda^I$ for every $\tau \in \Irr(I/L_1)$, and so $\tau(1)\lambda_0(1)=\tau(1)$ divides $6$. However we can choose $\tau \in \Irr(I/L_1)$ with $\tau(1)=21$ and this degree does not divide $6$. Therefore $f_i>1$, for all $i$. We deduce that, for each $i$, $f_i$ is the degree of a nontrivial proper irreducible projective representation of $M_{22}$. As $\mu_i(1)=f_i \lambda(1)=f_i$, each $f_i$ divides $6$. However this is impossible as the smallest nontrivial proper projective degree of $M_{22}$ is $10$.\smallskip

\noindent \textbf{(7)} $H_{0}=Fi_{24}'$. Then by Lemma~\ref{lem:spor}(b), $U/M\cong 2.Fi_{23}$, and, for each $i$, $t\varphi_i(1)$ divides one of the numbers in $\mathcal{A}$
\begin{align*}
  \mathcal{A}:=\{&2^4\cdot5^2\cdot7\cdot17\cdot23, 2\cdot3^3\cdot7\cdot11\cdot13\cdot17, 2^2\cdot3\cdot11\cdot13\cdot17\cdot23,\\
  &2^3\cdot3\cdot7\cdot11\cdot13\cdot23,
   2^4\cdot3\cdot13\cdot17\cdot23, 2^2\cdot7\cdot11\cdot17\cdot23, 11\cdot13\cdot17\cdot23\}.
\end{align*}
By inspecting the list of maximal subgroups of $Fi_{23}$ in ~\cite[pp. 177-180]{Atlas}, the index of a maximal subgroup of $U/M$ divides no number in $\mathcal{A}$, then $t=1$, and so $I=U$. As the Schur multiplier of $I/M\cong Fi_{23}$ is trivial and $\theta$ is $I$-invariant, we deduce from~\cite[Theorem 11.7]{Isaacs-book}, that $\theta$ extends to $\theta_0\in \Irr(I)$. By ~\ref{lem:gal}(b), $\tau\theta_0$ is an irreducible constituent of $\theta^I$ for every $\tau\in \Irr(I/M)$, and so $\tau(1)\theta_0(1)=\tau(1)$ divides one of the numbers in $\mathcal{A}$. Choose $\tau\in \Irr(I/M)=\Irr(Fi_{23})$ with $\tau(1)=559458900$. This degree divides none of the numbers in $\mathcal{A}$, which is a contradiction.
\end{proof}

\begin{proposition}\label{prop:spor-4}
Let $G$ be a finite group with $\cd(G)=\cd(H)$ where $H$ is an almost simple group whose socle is a sporadic simple group $H_{0}$. If $G'/M$ is the chief factor of $G$, then $M=1$, and hence $G'\cong H_{0}$.
\end{proposition}
\begin{proof}
Here we deal with the groups mentioned in Remark~\ref{rem:main}, namely, $H=\Aut(H_{0})$, where $H_{0}$ is one of the sporadic groups $J_{2}$, $J_{3}$, $McL$, $HS$, $He$, $HN$, $Fi_{22}$, $Fi_{24}'$, $O'N$ and $Suz$. It follows from Proposition~\ref{prop:spor-1-2} that $G'/M$ is isomorphic to $H_{0}$, and by Proposition~\ref{prop:spor-3}, every linear character $\theta$ of $M$ is $G'$-invariant. We no apply Lemma~\ref{lem:schur}, and conclude that $|M/M'|$ divides the order of Schur Multiplier $\M(H_{0})$, see Table~\ref{tbl:spor}. Therefore, $G'/M'$ is isomorphic to either $H_{0}$, or one of the groups in the third column of Table~\ref{tbl:spor-2}. In the latter case, we observe by \ATLAS \cite{Atlas} that $G'/M'$ has a degree as in the fifth column of Table~\ref{tbl:spor-2} which must divide some degrees of $H_{0}$, which is a contradiction. Therefore, $|M/M'|=1$, or equivalently, $M$ is perfect.

Suppose that $M$ is non-abelian, and let $N\leq M$ be a normal subgroup of $G'$ such that $M/N$ is a chief factor of $G'$. Then $M/N\cong S^{k}$, for some non-abelian simple group $S$. It follows from Lemma~\ref{lem:exten} that $S$ possesses a nontrivial irreducible character $\varphi$ such that $\varphi^{k}\in \Irr(M/N)$ extends to $G'/N$. By Lemma~\ref{lem:gal}(b), we must have $\varphi(1)^{k}\tau(1)\in \cd(G'/N)\subseteq \cd(G')$, for all $\tau \in \Irr(G'/M)$. Now we can choose $\tau\in G'/M$ such that $\tau(1)$ is the largest degree of $H_{0}$ as in the last column of Table~\ref{tbl:spor-2}, and since $\varphi$ is nontrivial, $\varphi(1)^{k}\tau(1)$ divides no degrees of $G$, which is a contradiction. Therefore, $M$ is abelian, and since $M=M'$, we conclude that $M=1$. Consequently, $G'$ is isomorphic to $H_{0}$.
\end{proof}

\begin{table}[t]
  \caption{Degrees of some groups related to sporadic simple groups $H_{0}$ in Proposition~\ref{prop:spor-4}.\label{tbl:spor-2}}
  \begin{tabular}{llp{2cm}lll}
  \hline
  $H_{0}$ &
  $\Aut(H_{0})$ &
  $G'/M'$ &
  Degree of &
  Degree of  &
  Largest degree of \\
   &
   &
   &
  $\Aut(H_{0})$ &
  $G'/M'$ &
  $H_{0}$\\
  \midrule
    $J_2$ &
    $J_2:2$ &
    $2.J_{2}$ &
    $28$ &
    $64$ &
    $336$ \\
    $HS$ &
    $HS:2$ &
    $2.HS$ &
    $308$ &
    $616$ &
    $3200$ \\
    $J_3$ &
    $J_3:2$ &
    $3.J_{3}$ &
    $170$ &
    $1530$ &
    $3078$ \\
    $McL$ &
    $McL:2$&
    $3.McL$ &
    $1540$ &
    $1980$ &
    $10395$ \\
    $He$ &
    $He:2$ &
    - &
    $102$ &
    - &
    $23324$\\
    $Suz$ &
    $Suz:2$ &
    $2.Suz$, $3.Suz$, $6.Suz$&
    $10010$ &
    $60060$ &
    $248832$ \\
    $O'N$ &
    $O'N:2$ &
    $3.O'N$ &
    $51832$ &
    $63612$ &
    $234080$\\
    $Fi_{22}$ &
    $Fi_{22}:2$ &
    $2.Fi_{22}$, $3.Fi_{22}$, $6.Fi_{22}$ &
    $277200$ &
    $235872$ &
    $2729376$ \\
    $HN$ &
    $HN:2$  &
    - &
    $266$ &
    - &
    $5878125$ \\
    $Fi'_{24}$ &
    $Fi'_{24}:2$  &
    $3.Fi_{24}'$ &
    $149674800$ &
    $216154575$ &
    $336033532800$ \\
  \bottomrule
  \multicolumn{6}{l}{The symbol `-' means that there is only one possibility for  $G'/M'$ which is $H_{0}$.}\\
  \end{tabular}
\end{table}

\begin{proposition}\label{prop:spor-5}
Let $G$ be a finite group with $\cd(G)=\cd(H)$ where $H$ is an almost simple group whose socle is a sporadic simple group $H_{0}$. Then $G/Z(G)$ is isomorphic to $H$.
\end{proposition}
\begin{proof}
By Remark~\ref{rem:main}, we will consider the case where $H=\Aut(H_{0})$ with $H_{0}$  one of the sporadic groups $J_{2}$, $J_{3}$, $McL$, $HS$, $He$, $HN$, $Fi_{22}$, $Fi_{24}'$, $O'N$ and $Suz$. According to  Proposition~\ref{prop:spor-4}, $G'$ is isomorphic to $H_{0}$. Let $A:= C_G(G')$. Since $G'\cap A=1$ and $G'A\cong G' \times A$, it follows that $G'\cong G'A/A\unlhd G/A\leq \Aut(G')$. Thus $G/A$ is isomorphic to $H_{0}$ or $\Aut(H_{0})=H_{0}:2$. In the case where $G/A$ is isomorphic to $H_{0}$, we must have $G\cong A\times H_{0}$. This is impossible as $G$ possesses a character of degree as in the fourth column of Table~\ref{tbl:spor-2}, however,  $H_{0}$ has no such degree. Therefore, $G/A$ is isomorphic to $\Aut(H_{0})$. Note also that $G'\cap A=1$. Then $[G,A]=1$, and hence $A=Z(G)$, as claimed.
\end{proof}





\end{document}